\documentclass[12pt,a4paper]{article}
\usepackage{amsmath,amssymb,amsfonts}
\usepackage{theorem}
\usepackage[dvips]{graphicx}
\newtheorem{thm}{Theorem}

\newtheorem{lm}{Lemma}
\newtheorem{rem}{Remark}
\newenvironment{proof}[1][{\it Proof.}]
 {\begin{trivlist}\item[\hskip\labelsep{\bfseries #1}]}{\end{trivlist}}
\newcommand\QED
 {\hfill\hbox{\vrule width 4pt height 6pt depth 1.5pt}\par\bigskip}
\usepackage{array}
\def\thline{\noalign{\hrule height 1pt}}

\begin{document}

\noindent 
DISTRIBUTIONS OF THE LARGEST SINGULAR VALUES OF SKEW-SYMMETRIC RANDOM
 MATRICES AND THEIR APPLICATIONS TO PAIRED COMPARISONS
\vskip 3mm

\vskip 5mm
\noindent Satoshi Kuriki

\noindent The Institute of Statistical Mathematics %

\noindent 10-3 Midoricho, Tachikawa, Tokyo 190-8562, Japan

\noindent  kuriki@ism.ac.jp

\vskip 3mm

\noindent ABSTRACT

Let $A$ be a real skew-symmetric Gaussian random matrix
whose upper triangular elements are independently distributed according to
the standard normal distribution.
We provide the distribution of the largest singular value
 $\sigma_1$ of $A$.
Moreover, by acknowledging the fact 
that the largest singular value can be
regarded as the maximum of a Gaussian field, we deduce the distribution of
the standardized largest singular value
$\sigma_1/\sqrt{\mathrm{tr}(A'A)/2}$.
These distributional results are utilized in Scheff\'{e}'s
paired comparisons model.  We propose tests
for the hypothesis of subtractivity based on the 
largest singular value of the skew-symmetric residual matrix.
Professional baseball league data are analyzed
 as an illustrative example. %
\vskip 4mm

\noindent 1. INTRODUCTION

Let $A=(a_{ij})$ be a $p\times p$ real skew-symmetric Gaussian matrix
whose upper triangular elements $a_{ij}$ $(1\le i<j\le p)$ are independently
distributed according to the standard normal distribution.
The density function of $A$ is given by
\begin{align}
 \frac{1}{(2\pi)^{p(p-1)/4}} \exp\Bigl\{ -\frac{1}{4}
 \mathrm{tr}(A'A) \Bigr\} \, dA, \quad dA=\prod_{i<j} da_{ij}.
\label{DENSITY}
\end{align}
The singular value decomposition of $A$ is given by
\begin{align}
\label{SVD0}
 A = \sum_{i=1}^t \sigma_i \bigl(u_{2i-1}u_{2i}'-u_{2i}u_{2i-1}'\bigr),
\end{align}
where $\sigma_1\geq\cdots\geq\sigma_t\geq 0$, $t=[p/2]$
 (the integer part of $p/2$), are the nonnegative singular values, and
$u_i$ is the $i$th column vector of a $p\times p$ orthogonal matrix $U$.
In this paper, we derive the distributions of $\sigma_1$
and its standardized version
$\sigma_1/\sqrt{\sum_{i=1}^t \sigma_i^2}$,
which are the largest singular values
 of the skew-symmetric matrices $A$ and $A/\sqrt{\mathrm{tr}(A'A)/2}$, respectively.

These distributional results are utilized in the analysis of
 paired comparisons.  Suppose that there are $m$ objects %
 (treatments, stimuli, {\it etc.\/}) $O_1,\ldots,O_m$ and
 that paired comparisons are made for all 
${m \choose 2}$ pairs.
For each $i<j$, $y_{ij}$ is assumed to be the observed
degree of preference of $O_i$ over $O_j$.
The observation $(y_{ij})$ is written as an $m\times m$ skew-symmetric
matrix by letting $y_{ji}=-y_{ij}$ and $y_{ii}=0$.
For such data, Scheff\'{e} (1952) proposed an analysis of variance based on the
following linear model:
\begin{align}
 y_{ij}   = \mu_{ij}+\varepsilon_{ij}, \qquad
 \mu_{ij} = (\alpha_i-\alpha_j)+\gamma_{ij} \qquad
 (1\leq i,j\leq m),
\label{MODEL}
\end{align}
where $\varepsilon_{ij}$ $(i<j)$ are independently distributed according to
the normal distribution $N(0,\sigma^2)$
with the mean $0$ and the variance $\sigma^2$.
The parameters $\alpha_i$ and $\gamma_{ij}$
are called the main effect and the interaction, respectively.
When the no interaction hypothesis $H_0 : \gamma_{ij}\equiv 0$ is true,
the model is easily interpreted since the relative preference
$\mu_{ij}$ is the difference of the scores $\alpha_i$ and $\alpha_j$.
For this reason, $H_0$ is called the hypothesis of subtractivity.
In this paper, we propose the use of the largest singular value
of the interaction estimator matrix $\bigl(\widehat\gamma_{ij}\bigr)$
as the test statistics for testing $H_0$ in the following two cases:
(i) the error variance $\sigma^2$ is known or an independent estimator
 $\widehat\sigma^2$ is available, and
(ii) $\sigma^2$ is unknown and no independent $\widehat\sigma^2$
is available.
The proposed tests are max-type statistics which suggest the direction of
discrepancy with $H_0$ when it is rejected.

This paper is arranged as follows.
In Section 2, the joint distribution of
singular values $(\sigma_1,\ldots,\sigma_t)$ of 
$A$, and the marginal distribution of 
$\sigma_1$ are given.
Moreover, acknowledging the fact 
that $\sigma_1$ can be regarded as
the maximum of a Gaussian random field on a manifold, the distribution of
 $\sigma_1/\sqrt{\sum_{i=1}^t \sigma_i^2}$ is derived by modifying that of
 $\sigma_1$ using the tube method (Kuriki and Takemura (2001),
 Takemura and Kuriki (2002), Adler and Taylor (2007)).
In Section 3,
test statistics for $H_0$
based on the largest singular values are proposed.
Furthermore, professional baseball league data are analyzed
as an illustrative example.

\vskip 4mm

\noindent 2. DISTRIBUTION OF THE LARGEST SINGULAR VALUES

\vskip 1mm

\noindent 2.1. Singular value decomposition

The set of $p\times p$ real skew-symmetric matrices is denoted by $Skew(p)$.
Let $A=(a_{ij})\in Skew(p)$ be a Gaussian random matrix 
whose upper triangular elements $a_{ij}$ $(i<j)$ are
independently distributed according to the standard normal distribution.
The density of $A$ is given in (\ref{DENSITY}).
The singular value decomposition (\ref{SVD0}) is rewritten as
\begin{alignat}{2}
 A=U D_\sigma U', \qquad
 D_\sigma
 &=\mathrm{diag}(\sigma_1 J,\ldots,\sigma_t J) & \quad & \mbox{if $p=2t$},
 \nonumber \\
 &= \mathrm{diag}(\sigma_1 J,\ldots,\sigma_t J,0) & \quad & \mbox{if $p=2t+1$},
\label{SVD}
\end{alignat}
where $\sigma=(\sigma_1,\ldots,\sigma_t)$,
 $\sigma_1\geq\cdots\geq\sigma_t\geq 0$, $t=[p/2]$,
is a vector of nonnegative singular values,
$J = \begin{pmatrix} 0 & 1 \\ -1 & 0 \end{pmatrix}$,
and $U\in O(p)$, the set of $p\times p$ orthogonal matrices.

In (\ref{SVD}), $U$ is not determined uniquely
 as an element of $O(p)$ since $A=(UH) D_\sigma (UH)'$
holds for any $H\in H(p)$, where
\begin{alignat*}{2}
 H(p)
  &= \{\mathrm{diag}({H}_1,\ldots,{H}_t) \mid {H}_i\in SO(2)\}
        & \quad & \mbox{if $p=2t$}, \nonumber \\
  &= \{\mathrm{diag}({H}_1,\ldots,{H}_t,e) \mid {H}_i\in SO(2),\, e=\pm 1 \}
        & \quad & \mbox{if $p=2t+1$}, 
\end{alignat*}
is a subgroup of $O(p)$.  Here, $SO(2)$ denotes the set of $2\times 2$
orthogonal matrices with the determinant $1$.
Conversely, $U$ in (\ref{SVD}) is defined uniquely as an element of
the left quotient space $U(p)=\{ U H(p) \mid U\in O(p) \}$
 $(=O(p)/H(p), \mbox{say})$ when all the singular values $\sigma_1,\ldots,\sigma_t$ are
distinct and positive, which is the case with probability $1$ in our application.
By introducing a quotient topology,
$U(p)$ becomes a manifold of the dimension $p(p-1)/2-t$.

The Jacobian of the transformation of (\ref{SVD}) was
given by Lemma 2 of Khatri (1965) as
\begin{align}
 dA = \prod_{i=1}^t \sigma_i^{2\epsilon}
            \prod_{i<j} \bigl(\sigma_i^2-\sigma_j^2\bigr)^2 \, d\sigma \, dU,
\label{JACOBIAN}
\end{align}
where
$\epsilon = p - 2t$ $(=0 \mbox{ for $p$ even},\ =1 \mbox{ for $p$ odd})$,
$dA=\prod_{i<j} da_{ij}$, $d\sigma=\prod_{i=1}^t d\sigma_i$, and
\begin{align}
 dU = \bigwedge_{(i,j)\in I_p} u_j' du_i
\label{WEDGE}
\end{align}
with $U=(u_1,\ldots,u_p)$, %
$
 I_p = \{ (i,j) \mid 1\leq i<j\leq p\} \setminus
 \{ (2h-1,2h) \mid 1\leq h\leq t\}
$.

The differential form $dU$ in (\ref{WEDGE}) is well-defined on $U(p)$
because $dU$ is independent of the choice of $\{{u}_i\}$, that is, 
\begin{align}
  \bigwedge_{(i,j)\in I_p} u_j' du_i
 =\bigwedge_{(i,j)\in I_p} \widetilde{u}_j' d\widetilde{u}_i
\label{CHOICE}
\end{align}
holds for
$ (\widetilde{u}_1,\ldots,\widetilde{u}_p)
 =(u_1,\ldots,u_p)H$, $H\in H(p)$.
(The proof of (\ref{CHOICE}) is similar to (2)
 in Section 4.6 of James (1954), and is omitted.)
Moreover, $dU$ is invariant with respect to the
orthogonal transformation $U\mapsto Q U$, $Q \in O(p)$.
(The proof is similar to (3) in Section 4.6 of
James (1954), and is omitted.)
The volume $\mathrm{Vol}(U(p))=\int_{U(p)}dU$ of the manifold $U(p)$
is needed to determine the normalizing constant of the density
 of $\sigma$.  This is calculated as follows.
\begin{lm}
\label{lm:vol-U}
\begin{alignat*}{2}
 \mathrm{Vol}(U(p))
=  \frac{2^t \pi^{p(p-1)/4}}
  {\prod_{i=1}^t \Gamma(p/2-i+1)\Gamma(p/2-i+1/2)}
 &= \frac{2^t \pi^{p(p-1)/4}}
  {\prod_{i=1}^p\Gamma(i/2)} & \qquad & \mbox{for $p$ even}, \\
 &= \frac{2^t \pi^{p(p-1)/4}}
  {\prod_{i=2}^p\Gamma(i/2)} & \qquad & \mbox{for $p$ odd},
\end{alignat*}
where $t=[p/2]$.
\end{lm}

\begin{proof}
Similar to the proof of Theorem 5.1 of James (1954)
obtaining the volume of the Stiefel manifold,
we can prove the recurrence relation
$$
 \mathrm{Vol}(U(p))=\mathrm{Vol}\bigl(\widetilde G(2,p)\bigr)\,\mathrm{Vol}(U(p-2))
$$
with $\mathrm{Vol}(U(2))=2$, $\mathrm{Vol}(U(1))=1$,
and 
$$ \mathrm{Vol}\bigl(\widetilde G(2,p)\bigr) = 
 \int_{\widetilde G(2,p)}\bigwedge_{i=1}^2\bigwedge_{j\geq 3} u_j' du_i
$$
is the volume of the oriented Grassmann manifold $\widetilde G(2,p)$
 (Appendix A.1).
The results follow from the fact that
$$ \mathrm{Vol}\bigl(\widetilde G(2,p)\bigr) 
 = 2 \frac{\Omega_p \Omega_{p-1}}{\Omega_2 \Omega_1}
$$
with $\Omega_n=2\pi^{n/2}/\Gamma(n/2)$ ((5.23) of James (1954)).
\QED
\end{proof}

\vskip 4mm

\noindent 2.2. Distributions of the largest singular value

Substituting (\ref{SVD}) and (\ref{JACOBIAN}) into (\ref{DENSITY}), and
integrating (\ref{DENSITY}) with respect to $U$ over $U(p)$,
we have the joint density of $\sigma_1>\cdots>\sigma_t>0$ as
\begin{align}
 c_p \exp\Bigl\{ -\frac{1}{2}\sum_i \sigma_i^2 \Bigr\}
 \prod_i \sigma_i^{2\epsilon}
 \prod_{i<j}\bigl(\sigma_i^2-\sigma_j^2\bigr)^2,
\label{LDENSITY}
\end{align}
where $c_p=\mathrm{Vol}(U(p))/(2\pi)^{p(p-1)/4}$ is the normalizing constant.
The integration of the joint density (\ref{LDENSITY}) over
$x>\sigma_1>\cdots>\sigma_t>0$ yields
the cumulative distribution function of the largest singular value
 $P\bigl(\sigma_1 < x \bigr)$.
Since the linkage factor in (\ref{LDENSITY}) is written as
the Vandermonde determinant
$ \prod_{i<j}\bigl(\sigma_i^2-\sigma_j^2\bigr)
 = \det\bigl(\sigma_j^{2(t-i)}\bigr)_{1\le i,j\le t} $, 
we have
\begin{align*}
P\bigl(\sigma_1 < x \bigr)
&= c_p \int_{x>\sigma_1>\cdots>\sigma_t>0}
 \det\biggl(\sum_{k=1}^t \sigma_k^{2(t-i)} \sigma_k^{2(t-j)}\biggr)\,
 \prod_{k=1}^t \sigma_k^{2\epsilon} \, e^{-\sigma_k^2/2} \, d\sigma_k \\
&= c_p \det\biggl(\int_0^x \sigma^{2(t-i)+2(t-j)+2\epsilon} \,
 e^{-\sigma^2/2} \, d\sigma \biggr)_{1\le i,j\le t}.
\end{align*}
The last equality in the expression above follows from
the determinental Binet-Cauchy formula
 ((2.1) of Krishnaiah (1976), (2.12) of Karlin and Rinott (1988)).
Making a change of variable, we obtain the following theorem.
\begin{thm}
\label{thm:largest-sv}
The distribution function of the largest singular value $\sigma_1$ is given by
\begin{align}
\label{largest-sv}
P\bigl(\sigma_1 < x \bigr)
= d_p \det\biggl(
    \int_0^{x^2} \phi^{p-i-j-1/2} e^{-\phi/2} \, d\phi
  \biggr)_{1\leq i,j\leq t},
\end{align}
where $t=[p/2]$, 
\begin{align*}
d_p = \frac{c_p}{2^t}
 &= \frac{1}{2^{p(p-1)/4}\,\prod_{i=1}^p\Gamma(i/2)}\quad\mbox{for $p$ even},\\
 &= \frac{1}{2^{p(p-1)/4}\,\prod_{i=2}^p\Gamma(i/2)}\quad\mbox{for $p$ odd}.
\end{align*}
\end{thm}

\begin{rem}
\label{rem:complex-wishart}
Let $\phi_1\ge\cdots\ge\phi_t\ge 0$ be the eigenvalues of a $t\times t$
 central complex Wishart Hermitian matrix
 $CW_t(t+\epsilon-1/2,I_t)$ with $\epsilon=p-2t$.
Then, it is observed that the joint density of $\sigma_i^2/2$ $(1\leq i\leq t)$ %
coincides with that of $\phi_i$ $(1\leq i\leq t)$.
(See (102) of James (1964).)
Accordingly, the marginal distribution of the largest eigenvalue of the complex
Wishart matrix obtained by Khatri (1964) is consistent with
Theorem \ref{thm:largest-sv}.
\end{rem}

\vskip 4mm

\noindent 2.3. Upper probability of the standardized largest singular value

Assume that the linear space of $p\times p$ real skew-symmetric matrices
$Skew(p)$ is endowed with the metric
$\langle A,B\rangle = \mathrm{tr}(A'B)/2 = \sum_{i<j} a_{ij} b_{ij}$,
$A=(a_{ij}),\,B=(b_{ij})\in Skew(p)$.
Let
$$ V(2,p) = \bigl\{ (h_1,h_2) : p\times 2 \mid
 h_1'h_1 = h_2'h_2 = 1,\,h_1'h_2 = 0 \bigr\}
$$
be the set of $p\times 2$ orthogonal matrices, that is, a Stiefel manifold.
The largest singular value of a skew-symmetric matrix $A$ is written as
\begin{align*}
\sigma_1
&= \max_{(h_1,h_2)\in V(2,p)} h_1' A h_2
= \max \, \bigl(h_1' A h_2 - h_2' A h_1\bigr)/2 \\
&= \max \, \mathrm{tr} \bigl\{ \bigl(h_1 h_2' - h_2 h_1'\bigr)'A \bigr\}/2
= \max_{H\in M} \, \langle H,A\rangle,
\end{align*}
where
\begin{align}
\label{M}
 M = M(p) = \bigl\{ h_1 h_2' - h_2 h_1' \mid (h_1,h_2)\in V(2,p) \bigr\}
  \subset Skew(p).
\end{align}
Let $n=\dim Skew(p)=p(p-1)/2$. 
It is evident that $M$ is a subset of the unit sphere
$ \{ A \in Skew(p) \mid \langle A,A\rangle=1 \} $
of $Skew(p)$ with the dimension $n-1$.
Moreover, as shown in Appendix A.1,
$M$ is diffeomorphic to an oriented Grassmann manifold
$\widetilde G(2,p-2)=V(2,p)/SO(2)$ with the dimension 
$$ d = \dim M = \dim \widetilde G(2,p-2) = \dim V(2,p) - \dim SO(2)
 = 2(p-2).
$$

The density (\ref{DENSITY}) of $A$ is rewritten as
$$ \frac{1}{(2\pi)^{n/2}}
  \exp\Bigl\{ -\frac{1}{2} \langle A,A\rangle \Bigr\} \, dA, $$
where $n=\dim Skew(p)$ and
 $dA=\prod_{i<j} da_{ij}$ is the volume element of $Skew(p)$ at $A$
induced by the metric $\langle\cdot,\cdot\rangle$.
This means that the distribution of $A$ is the standard
multivariate normal distribution in $Skew(p)$.
According to the general theory of the tube method
 (Kuriki and Takemura (2001), Takemura and Kuriki (2002)),
there exist coefficients $w_{d+1}, w_{d-1}, \ldots$,
referred to as Weyl's geometric invariants, such that
\begin{align}
 P\bigl(\sigma_1 > x \bigr)
&= P\Bigl( \max_{H\in M} \, \langle H,A\rangle > x \Bigr)
 \nonumber \\
&= \sum_{k=0}^{[d/2]} w_{d+1-2k} \overline G_{d+1-2k}(x^2)
 + o\bigl(e^{-x^2/2}\bigr), \quad x\to\infty,
\label{tail}
\end{align}
and
\begin{align}
 P\Biggl( \frac{\sigma_1}{\sqrt{\sum_{i=1}^t \sigma_i^2}} > x \Biggr)
&= P\biggl( \frac{\max_{H\in M} \, \langle H,A\rangle}
  {\sqrt{\langle A,A\rangle}} > x \biggr)
 \nonumber \\
&= \sum_{k=0}^{[d/2]} w_{d+1-2k}
  \overline B_{(d+1-2k)/2,(n-d-1+2k)/2} (x^2),
 \quad x\ge \cos\theta_c
\label{tube}
\end{align}
hold, where $\theta_c>0$ is a geometric quantity of $M$ called
 the critical radius,
$\overline G_{\nu}(\cdot)$ 
is the upper probability of the chi-square distribution
with $\nu$ degrees of freedom, and
$\overline B_{a,b}(\cdot)$ 
is the upper probability of the beta distribution with the parameter $(a,b)$.

Let
$ C = \bigl(2^{p-i-j+1/2}\,\Gamma(p-i-j+1/2)\bigr)_{1\le i,j\le t}$, %
$t=[p/2]$.
From Theorem \ref{thm:largest-sv}, we have
\begin{align*}
P\bigl(\sigma_1 > x \bigr)
&= 1-\det(C)^{-1} \det\biggl(\int^{x^2}_0
  \phi^{p-i-j-1/2} e^{-\phi/2} \, d\phi \biggr)_{1\le i,j\le t} \\
&= 1-\det(C^{-1}) \det\biggl(C - \biggl( \int_{x^2}^\infty
  \phi^{p-i-j-1/2} e^{-\phi/2} \, d\phi \biggr)_{1\le i,j\le t} \biggr) \\
&= 1-\det\biggl(I - C^{-1} \biggl( \int_{x^2}^\infty
  \phi^{p-i-j-1/2} e^{-\phi/2} \, d\phi \biggr)_{1\le i,j\le t} \biggr).
\end{align*}
Noting that
$$ \int_{x^2}^\infty \phi^{\nu/2-1} e^{-\phi/2} \, d\phi
 = 2^{\nu/2}\,\Gamma(\nu/2)\,\overline G_\nu(x^2)
 = O\bigl(x^{\nu-2} e^{-x^2/2}\bigr), \quad x\to\infty, $$
we have
\begin{align}
P(\sigma_1 >x)
&= \mathrm{tr}\biggl(C^{-1} \biggl( \int_{x^2}^\infty
  \phi^{p-i-j-1/2} e^{-\phi/2} \, d\phi \biggr)_{1\le i,j\le t} \biggr)
   + o\bigl(e^{-x^2/2}\bigr) \nonumber \\
&= \sum_{i,j=1}^t g^{ij} g_{ij}\,\overline G_{2p-2i-2j+1}(x^2)
   + o\bigl(e^{-x^2/2}\bigr),
\label{tail2}
\end{align}
where
$ g_{ij} = \Gamma(p-i-j+1/2) $
and $g^{ij}$ is the $(i,j)$th element of the inverse matrix of
$(g_{ij})_{1\le i,j\le t}$.
Comparing (\ref{tail2}) and (\ref{tail}), we obtain the theorem below.
\begin{thm}
\label{thm:tube}
When $p\ge 4$, the upper probability of
the standardized largest singular value is given by
\begin{align}
\label{tube2}
P\Biggl(\frac{\sigma_1}{\sqrt{\sum_{i=1}^t \sigma_i^2}} > x \Biggr)
= \sum_{i,j=1}^t g^{ij} g_{ij}\,
  \overline B_{(2p-2i-2j+1)/2,(n-2p+2i+2j-1)/2}(x^2), \quad x\ge 1/\sqrt{2},
\end{align}
with $n=p(p-1)/2$.
\end{thm}
\begin{proof}
Weyl's geometric invariants in (\ref{tail}) are given by
$ w_{d+1-2k} = w_{2p-3-2k}
 = \sum_{i+j=k+2} g^{ij} g_{ij} $.
Since this asymptotic expansion is uniquely represented, we have (\ref{tube2}) 
from (\ref{tube}).
The critical radius $\theta_c$ is proved to be $\pi/4$ in
 Appendix A.2.
\QED
\end{proof}

\begin{lm}
\label{lm:hankel}
The $(i,j)$th element of the inverse matrix of
 $(g_{ij})_{1\le i,j\le t}$ with $g_{ij}=\Gamma(p-i-j+1/2)$, $t=[p/2]$,
 is explicitly given as follows.
\begin{align}
g^{ij}
&= \frac{(-1)^{i+j}}
 {\Gamma(t+1-i)\,\Gamma(t+\epsilon+1/2-i)\,\Gamma(t+1-j)\,
  \Gamma(t+\epsilon+1/2-j)}
 \nonumber \\
&\qquad\times \sum_{k=1}^{\min(i,j)}
 \frac{\Gamma(t+1-k)\,\Gamma(t+\epsilon+1/2-k)}{\Gamma(i+1-k)\,\Gamma(j+1-k)},
 \quad \epsilon = p-2t.
\label{hankel}
\end{align}
\end{lm}

\smallskip
A sketch of the proof of Lemma \ref{lm:hankel} is provided
 in Appendix A.3.

Table I shows the upper probabilities at the critical point
 $P\bigl(\sigma_1/\sqrt{\sum \sigma_i^2} > 1/\sqrt{2} \bigr)$.
Note that when $p=4$ or $5$, $\sigma_1^2/(\sigma_1^2+\sigma_2^2)>1/2$ holds
with probability $1$.
This table shows that when $p$ is not so large, 
the formula given in Theorem \ref{thm:tube} provides a
wide range of values for the upper probabilities.

\begin{center}
Table I. Upper probabilities at the critical point $1/\sqrt{2}$.
\end{center}

\begin{center}
\begin{tabular}[h]{cc|cc|cc|cc}
\thline
$p$& prob. & $p$ & prob. & $p$ & prob.& $p$ & prob.\\
\hline
4  &  1.0000  &  8 &  0.9614 & 12 &  0.3236 & 16 & 0.0048 \\
5  &  1.0000  &  9 &  0.8827 & 13 &  0.1603 & 17 & 0.0009 \\
6  &  0.9989  & 10 &  0.7354 & 14 &  0.0634 & 18 & $<$0.0001 \\
7  &  0.9913  & 11 &  0.5328 & 15 &  0.0197 \\
\thline
\end{tabular}
\end{center}

\smallskip
\begin{rem}
According to the Gauss-Bonnet theorem
 (Lemma 3.5 of  Kuriki and Takemura (2001), %
 Corollary 3.1 of Takemura and Kuriki (2002)), 
the Euler-Poincar\'{e} characteristic of the
manifold $M$, or equivalently that of $\widetilde G(2,p-2)$, is given by
\begin{align*}
\chi(M) = \chi\bigl(\widetilde G(2,p-2)\bigr) &=  2 \sum_{k=1}^{[d/2]} w_{d+1-2k} 
  = 2 \sum_{i,j=1}^t g^{ij} g_{ij} = 2\,[p/2].
\end{align*}
\end{rem}

\vskip 4mm

\noindent 3. APPLICATIONS TO PAIRED COMPARISONS

\vskip 1mm

\noindent 3.1. Tests for the hypothesis of subtractivity

In the paired comparisons model (\ref{MODEL}),
we assume the side conditions
$\sum_i \alpha_i=0$ and $\sum_i \gamma_{ij}=\sum_j\gamma_{ij}=0$.
The least square estimators of $\alpha_i$ and $\gamma_{ij}$
are given by
\begin{align}
\label{ols}
 \widehat\alpha_i = \sum_{j=1}^m y_{ij}/m, \qquad%
 \widehat\gamma_{ij} = y_{ij} - \bigl(\widehat\alpha_i - \widehat\alpha_j\bigr),
\end{align}
respectively.
Scheff\'{e} (1952) showed that $\sum_{i<j} \widehat\gamma_{ij}^2/\sigma^2$
follows the chi-square distribution with
 $(m-1)(m-2)/2$ degrees of freedom when
the hypothesis of subtractivity $H_0 : \gamma_{ij}\equiv 0$ holds.
Based on this property, when $\sigma^2$ is known, or unknown but
there exists an independent estimator $\widehat\sigma^2$ of $\sigma^2$
made from the replication of the observations,
the chi-square or $F$ statistics for testing
$H_0$ can be constructed.

Let $\Gamma=(\gamma_{ij})$ and $\widehat\Gamma=\bigl(\widehat\gamma_{ij}\bigr)$.
The $i$th largest nonnegative singular value is denoted by
$\sigma_i(\cdot)$.
In this paper, instead of Scheff\'{e} (1952)'s ANOVA described above,
we propose test statistics
$\sigma_1\bigl(\widehat\Gamma\bigr)/\sigma$ when $\sigma^2$ is known,
or
$\sigma_1\bigl(\widehat\Gamma\bigr)/\widehat\sigma$ when $\sigma^2$ is unknown.
Because $\sigma_1\bigl(\widehat\Gamma\bigr)=0$
if and only if $\widehat\Gamma=0$, our proposed tests are consistent.
The following lemma about the distribution of the singular values
of $\widehat\Gamma$ can be easily proved.
The critical points of the proposed test
 are calculated by virtue of this lemma.
\begin{lm}
\label{lm:hatGamma}
Under the null hypothesis $H_0$,
the distribution followed by the nonzero singular values
$\sigma_i\bigl(\widehat\Gamma\bigr)$, $i=1,\ldots,[(m-1)/2]$,
is the same as that followed by the $(m-1)\times (m-1)$ skew-symmetric matrix $A$
whose density is (\ref{DENSITY}) with $p=m-1$.
\end{lm}

One advantage of our proposed test is that
the statistic $\sigma_1\bigl(\widehat\Gamma\bigr)$ provides us with some suggestions
about the direction of discrepancy with the null hypothesis $H_0$
when it is rejected.
Note that the largest singular value is
$\sigma_1\bigl(\widehat\Gamma\bigr) = \max_{c,d} c'\widehat\Gamma d$,
where the maximum is taken over $c=(c_1,\ldots,c_m)'$ and
$d=(d_1,\ldots,d_m)'$ such that
\begin{align}
\label{contrast}
 \sum_i c_i^2 =\sum_i d_i^2 =1,
 \qquad \sum_i c_i =\sum_i d_i =\sum_i c_i d_i =0.
\end{align}
When $H_0$ is rejected, we can examine the contrast functions
$c'\Gamma d$ such that $c'\widehat\Gamma d$ is large.
Although it is difficult to interpret the contrasts $c'\Gamma d$
in general, it must be noted that this class of contrasts
includes an interesting subclass of contrasts, as described below.
Let $c=(c_1,\ldots,c_m)'$ with $c_i=1/\sqrt{2}$, $c_j=-1/\sqrt{2}$,
 $c_l=0$ $(l\ne i,j)$, and let
$d=(d_1,\ldots,d_m)'$ with $d_i=d_j=1/\sqrt{6}$, $d_k=-2/\sqrt{6}$,
 $d_l=0$ $(l\ne i,j,k)$.  Then,
$$
 c'\Gamma d = (\gamma_{ij}+\gamma_{jk}+\gamma_{ki})/\sqrt{3}
 = (\mu_{ij}+\mu_{jk}+\mu_{ki})/\sqrt{3},
$$
which represents a departure
from the hypothesis of subtractivity among the three objects
$O_i$, $O_j$ and $O_k$.
This was introduced in the context of the Bradley-Terry model
and named as the three-way deadlock parameter by Hirotsu (1983).

So far, we have considered the case where $\sigma^2$ is known or there exists
an independent estimator $\widehat\sigma^2$.
When $\sigma^2$ is unknown and no estimator of $\sigma^2$ is available,
testing $H_0$ on the basis of model (\ref{DENSITY}) is no longer possible.
Instead, we assume a specific model for the interaction.
The following is an analogue to the test for interaction in the two-way layout
without replication proposed by Johnson and Graybill (1972).
\begin{thm}
\label{thm:lrt}
For the $m\times m$ paired comparisons data $(y_{ij})$, assume the model
\begin{align*}
 y_{ij}   &= \mu_{ij}+\varepsilon_{ij}, \nonumber \\
 \mu_{ij} &= (\alpha_i-\alpha_j) + \lambda (c_i d_j - d_i c_j),
 \quad\lambda\ge 0, \quad 1\leq i,j\leq m,
\end{align*}
where $c_i$ and $d_i$ satisfy (\ref{contrast}), and
$\varepsilon_{ij}$ $(i<j)$ are independently distributed according to
the normal distribution $N(0,\sigma^2)$ with $\sigma^2$ unknown.
Then, the likelihood ratio test statistic for testing the hypothesis $\lambda=0$ is
$$
\textstyle
 \sigma_1\bigl(\widehat\Gamma\bigr)\big/
  \sqrt{\sum_{i=1}^{[(m-1)/2]}\sigma_i^2\bigl(\widehat\Gamma\bigr)} =
 \sigma_1\bigl(\widehat\Gamma\bigr)\big/
  \sqrt{\mathrm{tr}\bigl(\widehat\Gamma'\widehat\Gamma\bigr)/2},
$$
that is, the standardized largest singular value of
 $\widehat\Gamma=\bigl(\widehat\gamma_{ij}\bigr)$,
 where $\widehat\gamma_{ij}$ is given in (\ref{ols}).   
\end{thm}

According to Lemma \ref{lm:hatGamma}, the upper probability
 ($p$-value) of the likelihood ratio statistic proposed in Theorem \ref{thm:lrt}
can be calculated by Theorem \ref{thm:tube} within the range
of Table I.

\vskip 4mm

\noindent 3.2. Analysis of professional baseball data

Table II gives the score sheet of the Central League,
 one of Japan's professional baseball leagues, in 1997.
For any pair of teams, the total number of games is $n=27$.
Let $r_{ij}$  $(1\leq i,j\leq m=6)$ be the number of games
where Team $i$ beats Team $j$,
as indicated in the $(i,j)$th cell of Table II.
There are no ties in this table, that is, $r_{ij}+r_{ji}=n$.

\begin{center}
Table II. Score sheet of the Central League in 1997.
\end{center}

\begin{center}
\begin{tabular}[h]{l|cccccc}
\thline
\ Team\ \,$i\ \backslash\ j$
    &  1&  2&  3&  4&  5&  6\\
\hline
1. Yakult    & -- & 13 & 15 & 19 & 20 & 16 \\
2. Yokohama  & 14 & -- & 16 & 13 & 10 & 19 \\
3. Hiroshima & 12 & 11 & -- & 13 & 12 & 18 \\
4. Yomiuri   &  8 & 14 & 14 & -- & 14 & 13 \\
5. Hanshin   &  7 & 17 & 15 & 13 & -- & 10 \\
6. Chunichi  & 11 &  8 &  9 & 14 & 17 & -- \\
\thline
\end{tabular}
\end{center}

\smallskip
When we suppose that all games are independent trials,
we can consider $r_{ij}$ to be a random variable following
the binomial distribution $\mathrm{Bin}(n,q_{ij})$,
where $q_{ij}$ is the probability that Team $i$ beats Team $j$.
In order to analyze this data based on the model (\ref{MODEL}),
we use the variance stabilizing transformation (Rao (1973), 6g.3)\,: %
$$
 f(q)=2\sqrt{n}\bigl(\sin^{-1}\sqrt{q}-\pi/4\bigr).
$$
Then, $f(r_{ij}/n)$ approximately follows
the normal distribution $N(f(q_{ij}),1)$.

In league games, however, the games in each pair are sometimes
arranged within a short time interval.  Due to this game design,
some serial correlations may occur.  If the correlation is positive,
the over dispersion is expected to be observed.  In such a situation,
$f(r_{ij})$ can be modeled as $N(f(q_{ij}),\sigma^2)$
 with the variance $\sigma^2$ unknown.

The test statistics for testing the subtractivity and their $p$-values
are as follows:\ %
the chi-square statistic $\mathrm{tr}\bigl(\widehat\Gamma'\widehat\Gamma\bigr)/2=15.765$
 ($\mbox{d.f.}=10$, $p\mbox{-value}=0.1066$),
the largest singular value $\sigma_1\bigl(\widehat\Gamma\bigr)=3.932$
 ($p\mbox{-value}=0.0543$), and
the standardized largest singular value
$\sigma_1\bigl(\widehat\Gamma\bigr)/ \allowbreak
 \sqrt{\sum_i\sigma_i\bigl(\widehat\Gamma\bigr)^2}=0.990$
 ($p\mbox{-value}=0.0348$).
 (The other nonzero singular value is $\sigma_2\bigl(\widehat\Gamma\bigr)=0.553$.)
  In any cases, the hypothesis of subtractivity
is found to be suspicious, or is rejected.
The maximum contrast of three-way deadlock is
\begin{align}
\label{gamma256}
 \max_{i<j<k,\ i>j>k}
 \bigl(\widehat\gamma_{ij}+\widehat\gamma_{jk}+\widehat\gamma_{ki}\bigr)/\sqrt{3}
 = \bigl(\widehat\gamma_{65}+\widehat\gamma_{52}+\widehat\gamma_{26}\bigr)/\sqrt{3}=2.832.
\end{align}

Figure I is the residual plot, which was introduced
for the Bradley-Terry model by Takeuchi and Fujino (1988).
In this figure, $m$ points
$(\sqrt{\sigma_1}u_i,\sqrt{\sigma_1}v_i)$ $(1\leq i\leq m)$ are plotted,
where $\sigma_1$ is the largest singular value of  $\widehat\Gamma$
 (the residual matrix under the null hypothesis), and
$u=(u_1,\ldots,u_m)'$ and $v=(v_1,\ldots,v_m)'$ are the eigenvectors
that correspond to $\sigma_1$.
This plot is based on the two-rank approximation
$\widehat\Gamma\simeq\sigma_1 (u v'-v u')$.
This idea of plotting a skew-symmetric matrix
by two-rank approximation originates from Gower (1977).
It can be observed that the triplet teams 6, 5 and 2, which gives the maximum
contrast of three-way deadlock, form a large triangle around the origin
in the counterclockwise direction.

Note that if $\widehat\Gamma = \sigma_1 (u v'-v u')$ holds exactly,
the signed area $S_{ijk}$ of the triangle with the apexes
$(\sqrt{\sigma_1}u_l,\sqrt{\sigma_1}v_l)$ $(l=i,j,k)$ %
in the counterclockwise direction satisfies
$2 S_{ijk}/\sqrt{3}
=\bigl(\widehat\gamma_{ij}+\widehat\gamma_{jk}+\widehat\gamma_{ki}\bigr)/\sqrt{3}$.
In our example, $2 S_{652}/\sqrt{3}=2.839$, which is very close to
the maximum value 2.832 in (\ref{gamma256}).

\bigskip
\begin{center}
\scalebox{0.725}{\includegraphics{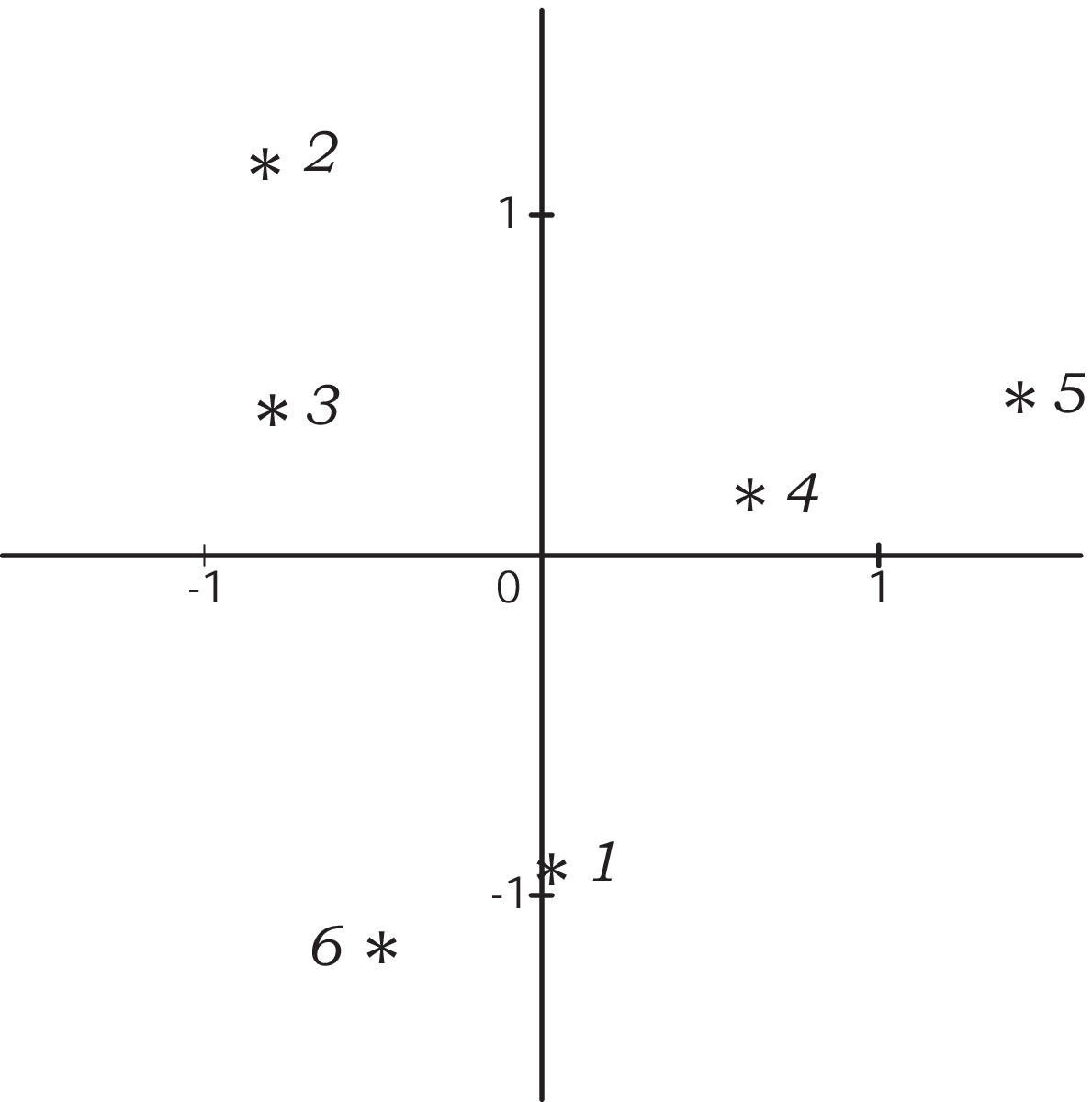}} %
\end{center}

\begin{center}
Figure I. Residual plot by two-rank approximation.
\end{center}

\vskip 4mm

\noindent APPENDIX

\vskip 1mm

\noindent A.1. The diffeomorphism between $M$ and $\widetilde G(2,p-2)$

Define the group action of $SO(p)$ on the space $Skew(p)$ by
$ X \mapsto H X H' \in Skew(p) $
for $X\in Skew(p)$, $H\in SO(p)$.
This action is obviously of $C^\infty$. %
Let
$$ X_0 = \begin{pmatrix} J & 0 \\
                         0 & 0 \end{pmatrix} \in Skew(p) ,\quad
     J = \begin{pmatrix} 0 & 1 \\
                        -1 & 0 \end{pmatrix}. $$
Then, the orbit that $X_0$ belongs to is given by
\begin{align*}
 SO(p)(X_0)
 &= \{ H X_0 H' \in Skew(p) \mid H\in SO(p) \}  \\
 &= \bigl\{ h_1 h_2' - h_2'h_1 \mid
  h_1'h_1 = h_2'h_2 = 1,\ h_1'h_2 = 0 \bigr\},
\end{align*}
which is the index manifold $M$ in (\ref{M}).
The isotropy subgroup of $X_0$ is
\begin{align*}
 SO(p)_{X_0}
 &= \{ H\in SO(p) \mid H X_0 H'=X_0 \}
  = \{ H\in SO(p) \mid H X_0 = X_0 H \}  \\
 &= \biggl\{ \begin{pmatrix} H_1 & 0 \\ 0 & H_2 \end{pmatrix} %
     \mid H_1\in SO(2),\ H_2\in SO(p-2) \biggr\} \\
 &= SO(2)\times SO(p-2),\ \mbox{say}. 
\end{align*}
Let $SO(p)/SO(p)_{X_0}=\{ H SO(p)_{X_0} \mid H\in SO(p) \}$
 be the quotient space endowed with the quotient topology. %
Define a map $f:SO(p)/SO(p)_{X_0}\to SO(p)(X_0)$ by
 $f(H SO(p)_{X_0})=H X_0 H'$.
Then, according to
 Theorem 4.3 and Corollary 4.4 %
of Kawakubo (1991), 
$f$ gives a $C^\infty$ embedding from $SO(p)/SO(p)_{X_0}$ to $SO(p)(X_0)$
(endowed with the relative topology). %
That is,
$SO(p)/SO(p)_{X_0} \cong SO(p)(X_0)=M$ $(C^\infty$-diffeomorphism).

$SO(p)/SO(p)_{X_0} = SO(p)/(SO(2)\times SO(p-2))$ is called
the oriented Grassmann manifold $\widetilde G(2,p-2)$.

\vskip 4mm

\noindent A.2. The critical radius of $M$

Let $t=(t^i)_{1\le i\le d}$, $d=2(p-2)$, denote the local coordinates
 of the submanifold $M$ in (\ref{M}) of $Skew(p)$
so that an element of $M$ is written as
$\phi = \phi(t) = h_1 h_2' - h_2 h_1' = H J H'$, where
$$
 H=H(t)= (h_1,h_2)\in V(2,p),
 \quad J = \begin{pmatrix} 0 & 1 \\ -1 & 0 \end{pmatrix}.
$$
The critical radius $\theta_c$ of $M$ is given by
$$
 \cot^2\theta_c = \sup_{\widetilde t \ne t}
 \frac{1-\bigl\langle \widetilde\phi, P_\phi\bigl(\widetilde\phi\bigr)\bigr\rangle}
      {(1-\bigl\langle \widetilde\phi,\phi\bigr\rangle)^2},
 \quad \widetilde\phi = \phi\bigl(\widetilde t\bigr),
$$
where $P_\phi : Skew(p)\to Skew(p)$ is the orthogonal projection
onto the subspace spanned by $\{\phi,\phi_1,\ldots,\phi_d\}$,
 $\phi_i = \partial \phi/\partial t^i$
 (Lemma 3.1 of Kuriki and Takemura (2001),
 Lemma 2.1 of Takemura and Kuriki (2002)).
Recall that $Skew(p)$ is a linear space of $p\times p$ real skew-symmetric
matrices endowed with the metric $\langle A,B\rangle=\mathrm{tr}(A'B)/2$.

Let $G=(g_{ij})=(\langle \phi_i,\phi_j\rangle)$, and
$G^{-1}=(g^{ij})$.
Noting that $\langle \phi,\phi\rangle = \mathrm{tr}(\phi'\phi)/2=1$
and hence
$\langle \phi,\phi_i\rangle = \mathrm{tr}(\phi'\phi_i)/2=0$,
we see that
$$
 \bigl\langle \widetilde\phi, P_\phi\bigl(\widetilde\phi\bigr)\bigr\rangle
 = \frac{1}{4} \bigl\{\mathrm{tr}\bigl(\widetilde\phi'\phi\bigr)\bigr\}^2
 + \frac{1}{4} \sum_{i,j=1}^d
     \mathrm{tr}\bigl(\widetilde\phi'\phi_i\bigr) \mathrm{tr}\bigl(\widetilde\phi'\phi_j\bigr) g^{ij}.
$$
Moreover,
$$
\bigl\{\mathrm{tr}\bigl(\widetilde\phi'\phi\bigr)\bigr\}^2
 = \bigl\{\mathrm{tr}\bigl(\widetilde H J'\widetilde H' H J H'\bigr)\bigr\}^2
 = \bigl\{\mathrm{tr}(R J R' J)\bigr\}^2
$$
with $\widetilde H=H\bigl(\widetilde t\bigr)$, $R = \widetilde H'H$,
and
$$
 \mathrm{tr}\bigl(\widetilde\phi'\phi_i\bigr)
 = \mathrm{tr}\bigl(\widetilde H J'\widetilde H' (H_i J H'+H J H_i')\bigr)
 = -2 \mathrm{tr}\bigl(J R' J \widetilde H'H_i\bigr)
$$
with $H_i=\partial H/\partial t^i$.
Note that $R = \widetilde H'H$ is a $2\times 2$ real matrix
such that the absolute values of the eigenvalues are less than or
equal to $1$, and $\widetilde\phi = \phi$ if and only if $R \in SO(2)$.
Since $H'H_i$ is skew-symmetric, we can put $H_i=b_i H J + \overline H C_i$ with
$b_i$ a scalar, $\overline H$ a $p\times (p-2)$ matrix
 such that $(H,\overline H)$ is $p\times p$ orthogonal,
and $C_i$ a $p\times 2$ matrix.  Therefore,
$$
 \mathrm{tr}\bigl(\widetilde\phi'\phi_i\bigr)
 = -2 \mathrm{tr}\bigl\{J R' J \widetilde H'(b_i H J + \overline H C_i)\bigr\}
 = -2 \mathrm{tr}\bigl(J R' J \widetilde H' \overline H C_i\bigr).
$$

On the other hand, as
\begin{align*}
g_{ij}
&= \langle \phi_i,\phi_j \rangle = \mathrm{tr}(\phi_i'\phi_j)/2 \\
&= \mathrm{tr}((H_i J' H'+H J' H_i')(H_j J H'+H J H_j'))/2 \\
&= \mathrm{tr}(H_i' H_j) - \mathrm{tr}(H'H_i J H'H_j J) 
= \{\mathrm{tr}(C_i'C_j) + 2 b_i b_j \} - 2 b_i b_j \\
&= \mathrm{tr}(C_i'C_j),
\end{align*}
we have
\begin{align*}
\sum_{i,j=1}^d & \mathrm{tr}\bigl(\widetilde\phi'\phi_i\bigr)
  \mathrm{tr}\bigl(\widetilde\phi'\phi_j\bigr) g^{ij}
= 4 \mathrm{tr}\bigl\{\bigl(J R' J \widetilde H' \overline H\bigr)
 \bigl(J R' J \widetilde H' \overline H\bigr)'\bigr\} \\
&= 4 \mathrm{tr}\bigl\{R' J \widetilde H' (I - H H') \widetilde H J' R\bigr\}
= 4 \mathrm{tr}(R R') + 4 \mathrm{tr}(R R' J R R' J).
\end{align*}

After summarizing the above, and conducting some further calculations, we obtain
\begin{align*}
 \cot^2\theta_c
&= \sup_{R \notin SO(2)}
 \frac{1 -\{\mathrm{tr}(RJR'J)\}^2/4 -\mathrm{tr}(R R') -\mathrm{tr}(R R' J R R' J)}
      {(1 + \mathrm{tr}(R J R' J)/2)^2} \\
&= \sup_{R\notin SO(2)} \Bigl\{
 1 - \frac{(r_{11}-r_{22})^2 + (r_{12}+r_{21})^2}
          {(1-r_{11}r_{22}+r_{12}r_{21})^2} \Bigr\},
\end{align*}
where $R=(r_{ij})_{1\le i,j\le 2}$. 

From the expression above, it holds $\cot^2\theta_c \le 1$ obviously.
On the other hand, when $p\ge 4$, $R=\widetilde H' H$ can be a zero matrix,
from which $\cot^2\theta_c=1$ or $\theta_c=\pi/4$ follows.

\vskip 4mm

\noindent A.3. A sketch of the proof of Lemma \ref{lm:hankel}

For $\delta>-1$, let $G=\bigl(\Gamma(\delta+2t-i-j+1)\bigr)_{1\le i,j\le t}$.
For $1\le k\le t-1$, let $B_k=(b_{k,ij})_{1\le i,j\le t}$ be
an upper band matrix with the $(i,j)$th element
$$ b_{k,ij} = \begin{cases}
 1             & \mbox{if $i=j$}, \\
 -(\delta+t-i) & \mbox{if $i+1=j$ and $i\le t-k$}, \\
 0             & \mbox{otherwise}.
 \end{cases}
$$
It can be confirmed that
$ B_{t-1}\cdots B_2 B_1 G = E T D$, where
$T=(t_{ij})$,
$$
 t_{ij} = \begin{cases} \displaystyle
 {t-j \choose t-i} & \mbox{if $i\ge j$}, \\
 0 & \mbox{if $i<j$}, \end{cases}
$$
is a lower triangular matrix, and
$$
 D = \mathrm{diag}\bigl( \Gamma(\delta+t-i+1) \bigr)_{1\le i\le t}, \quad
 E = \mathrm{diag}\bigl( (t-i)! \bigr)_{1\le i\le t}.
$$
The product of $B_k$ becomes an upper triangular matrix
$B_{t-1}\cdots B_2 B_1=B=(b_{ij})$ with
$$
 b_{ij} = \begin{cases} \displaystyle
 (-1)^{i+j}{t-i \choose t-j}
 \frac{\Gamma(\delta+t-i+1)}{\Gamma(\delta+t-j+1)} & \mbox{if $i\le j$}, \\
 0 & \mbox{if $i>j$}. \end{cases}
$$
The inverse matrix of $T$ becomes $T^{-1}=(t^{ij})$
 with $t^{ij}=(-1)^{i+j} t_{ij}$.

Then, the inverse matrix of $G$ is calculated as
 $G^{-1}=D^{-1} T^{-1} E^{-1} B$.
Setting $\delta=\epsilon-1/2$ yields (\ref{hankel}).

Further, note that
\begin{align}
\label{det}
 \det(G)=\det(D)\det(E)= \prod_{i=1}^t\Gamma(\delta+t-i+1)\,(t-i)!.
\end{align}
This identity (\ref{det}) provides another proof that
$d_p$ in (\ref{largest-sv}), Theorem \ref{thm:largest-sv}, 
is actually the normalizing constant.

\vskip 3mm

\noindent ACKNOWLEDGEMENT

The author is very grateful to Tomoko Nakao for her
 help in analyzing the baseball data.

\vskip 3mm

\noindent BIBLIOGRAPHY

\vskip 3mm

\noindent Adler, R.\,J.\ and Taylor, J.\,E.\ (2007).
{\it Random Fields and their Geometry\/}. New York: Springer.

\vskip 2mm

\noindent Gower, J.\,C.\ (1977).
The analysis of asymmetry and orthogonality.
In {\it Recent Developments in Statistics\/} %
 (J.\,R.\ Barra, F.\ Brodeau, G.\ Romier and B.\ Van Cutsem eds.), 109--123, Amsterdam: North-Holland.

\vskip 2mm

\noindent Hirotsu, C.\ (1983).
Defining the pattern of association in two-way contingency tables.
{\it Biometrika\/}, {\bf 70}, 579--589.

\vskip 2mm

\noindent James, A.\,T.\ (1954).
Normal multivariate analysis and the orthogonal group.
{\it Ann.\ Math.\ Statist.\/}, {\bf 25}, 40--75.

\vskip 2mm

\noindent James, A.\,T.\ (1964).
Distributions of matrix variates and latent roots derived from normal samples.
{\it Ann.\ Math.\ Statist.\/}, {\bf 35}, 475--501.

\vskip 2mm

\noindent Johnson, D.\,E.\ and Graybill, F.\,A.\ (1972).
An analysis of a two-way model with interaction and no replication.
{\it J.\ Amer.\ Statist.\ Assoc.\/}, {\bf 67}, 862--868.

\vskip 2mm

\noindent Karlin, S.\ and Rinott, Y.\ (1988).
A generalized Cauchy-Binet formula and applications to total positivity
 and majorization.
{\it J.\ Multivariate Anal.\/}, {\bf 27}, 284--299.

\vskip 2mm

\noindent Kawakubo, K.\ (1991).
{\it The Theory of Transformation Groups\/}. Oxford University Press: Oxford.

\vskip 2mm

\noindent Khatri, C.\,G.\ (1964).
Distribution of the largest or the smallest characteristic root under null
hypothesis concerning complex multivariate normal populations.
{\it Ann.\ Math.\ Statist.\/}, {\bf 35}, 1807--1810.

\vskip 2mm

\noindent Khatri, C.\,G.\ (1965).
A test for reality of a covariance matrix in a certain complex Gaussian
distribution. {\it Ann.\ Math.\ Statist.\/}, {\bf 36}, 115--119.

\vskip 2mm

\noindent Krishnaiah, P.\,R.\ (1976).
Some recent developments on complex multivariate distributions.
{\it J.\ Multivariate Anal.\/}, {\bf 6}, 1--30.

\vskip 2mm

\noindent Kuriki, S.\ and Takemura, A.\ (2001).
Tail probabilities of the maxima of multilinear forms and their
 applications. {\it Ann.\ Statist.\/}, {\bf 29}, 328--371.

\vskip 2mm

\noindent Rao, C.\,R.\ (1973).
{\it Linear Statistical Inference and its Applications\/}, 2nd ed.
New York: Wiley.

\vskip 2mm

\noindent Scheff\'{e}, H.\ (1952).
An analysis of variance for paired comparisons.
{\it J.\ Amer.\ Statist.\ Assoc.\/}, {\bf 47}, 381--400.

\vskip 2mm

\noindent Takemura, A.\ and Kuriki, S.\ (2002).
Maximum of Gaussian field on piecewise smooth domain:
 Equivalence of tube method and Euler characteristic method. %
{\it Ann.\ Appl.\ Probab.\/}, {\bf 12}, 768--796.

\vskip 2mm

\noindent Takeuchi, K.\ and Fujino, Y.\ (1988).
{\it Mathematical Sciences of Sports\/} (in Japanese).
Tokyo: Kyoritsu. %

\end{document}